\documentclass[12pt]{amsart}

\usepackage{latexsym}
\usepackage{amssymb}
\usepackage{amsmath}
\usepackage{enumerate}
\usepackage{amsmath,amsthm,amsfonts,amssymb,latexsym,xy}
\xyoption{all}
\usepackage{enumerate}

\newtheorem{theorem}{Theorem}
\newtheorem{lemma}[theorem]{Lemma}
\newtheorem{proposition}[theorem]{Proposition}
\newtheorem{corollary}[theorem]{Corollary}
\newtheorem{conjecture}[theorem]{Conjecture}
\theoremstyle{definition}
\newtheorem{definition}[theorem]{Definition}

\newtheorem{remark}[theorem]{Remark}
\newtheorem{question}[theorem]{Question}

\begin{document}

\author[Kate Juschenko]{Kate Juschenko}

\address{\tt EPFL, 1015 Lausanne Switzerland}
  \email{\tt kate.juschenko@gmail.com}

\author[Tatiana Nagnibeda]{Tatiana Nagnibeda}

\address{\tt
Section de math\'ematiques,
Universit\'e de Gen\`eve,
2-4, rue du Li\`evre c.p. 64,
1211 Gen\`eve, Switzerland}
\email{\tt tatiana.smirnova-nagnibeda@unige.ch}

\thanks{The authors acknowledge support of the Swiss National
  Foundation for Scientific Research and of the Mittag-Leffler Insitute}

\title[Spectral radius and percolation constants]{Small spectral radius and percolation constants on non-amenable Cayley graphs}


\begin{abstract}
Motivated by the Benjamini-Schramm non-unicity of percolation conjecture we study the following question. For a given finitely generated non-amenable group 
$\Gamma$,  does there exist a generating set $S$ such that the Cayley graph $(\Gamma,S)$, without loops and multiple edges, has non-unique percolation, i.e., $p_c(\Gamma,S)<p_u(\Gamma,S)$? We show that this is true if $\Gamma$ contains  an infinite normal subgroup $N$  such  that $\Gamma/ N$ is non-amenable. Moreover for any finitely generated group $G$ containing $\Gamma$ there exists a generating set $S'$ of $G$ such that $p_c(G,S')<p_u(G,S')$. In particular this applies to  free Burnside groups $B(n,p)$ with $n \geq 2, p \geq 665$.
We also explore how various non-amenability numerics, such as the isoperimetric constant and the spectral radius, behave on various growing generating sets in the group.
\end{abstract}

\date{}

\maketitle

\section{Introduction}
Let $\Gamma$ be an infinite finitely generated group and let $S$ be a finite symmetric generating set in $\Gamma$. 
The {\it isoperimetric constant} (or the {\it Cheeger constant})  of  $\Gamma$ with respect to $S$ is 
$$\phi(\Gamma, S)=\inf\limits_{F\subset \Gamma}\frac{\Sigma_{s\in S} |sF\backslash F|}{|F|},$$
where the infimum is taken over all finite subsets $F$ of $\Gamma$. 
Equivalently we can write 
$$\phi(\Gamma, S)=\inf\limits_{F\subset \Gamma}\frac{|\partial_E F|}{|F|},$$
where the infimum is taken over all finite subsets $F$ of $\Gamma$ and where $\partial_E F$ denotes the boundary of  $F$, i.e., the set of all edges connecting $F$ to its complement in the Cayley graph of $\Gamma$ with respect to $S$. 


The isoperimetric constant $\phi(\Gamma, S)$ normalized by the size of the generating set is often called the  {\it conductance} constant of $\Gamma$ with respect to $S$:
$$h(\Gamma, S)=\frac{1}{|S|}\phi(\Gamma, S).$$

Note that we can also define a different isoperimetric constant $\phi_V(\Gamma,S)$ by considering the boundary $\partial_V$ understood as the set of vertices at distance $1$ from $F$. We have 
$$\phi_V(\Gamma,S)\leq \phi(\Gamma, S)\leq |S|\phi_V(\Gamma, S) .$$

Let $\lambda:\Gamma\rightarrow B(l_2\Gamma)$ be the left-regular representation of $\Gamma$. The {\it spectral radius} of $\Gamma$ with respect to $S$ is 
$$\rho(\Gamma, S)=\frac{1}{|S|}||\Sigma_{g\in S}\lambda(g)||.$$
Then from \cite{Kesten} and \cite{Fol} we have the following characterization of amenability:

\begin{enumerate}
\item $\Gamma$ is amenable,
\item $\rho(\Gamma, S)=1$, for some (iff for every) finite generating set $S\subseteq \Gamma$,
\item $\phi(\Gamma,S)=\phi_V(\Gamma,S)=h(\Gamma, S)=0$, for some (iff for every) finite generating set $S\subseteq \Gamma$.
\end{enumerate}

Note that the conductance constant and the spectral radius are connected by the following inequalities (see \cite{Mohar}):
\begin{equation}
\frac{|S|(1-\rho(\Gamma,S))}{|S|-1}\leq h(\Gamma,S)\leq \sqrt{1-\rho(\Gamma,S)^2}.
\end{equation}
Equivalently 
\begin{equation}
1-\frac{h(\Gamma,S)(|S|-1)}{|S|}\leq \rho(\Gamma,S)\leq \sqrt{1-h(\Gamma,S)^2}.
\end{equation}

For $0\leq p\leq 1$ consider the {\it Bernoulli bond percolation} on the Cayley graph $(\Gamma, S)$. Namely, an edge $(g,sg)$ is open with probability $p$ and closed with probability $1-p$; connected components of the subgraph spanned by the open edges are called {\it open clusters}. By Kolmogorov's $0-1$ law, the probability of existence of an infinite open cluster is either $0$ or $1$, and there exists a {\it critical value} $p_c$ -- the smallest such that an infinite cluster exists almost surely for all $p>p_c$. The critical value can also be characterized with the help of the percolation function $\theta(p)$ defined as the probability that the origin $e\in \Gamma$ belongs to an infinite open cluster: 
$$p_c(\Gamma, S)=\sup \{p:\theta(p)=0\} .$$ A second critical value can be defined via $\xi(p)$, the probability that there exists exactly one infinite open cluster (\cite{HP}): 
$$p_u(\Gamma,S)=\inf \{p: \xi(p)=1\} .$$ In general, $$0<p_c\leq p_u\leq 1 .$$
In the amenable case, the two critical values coincide (\cite{BK}), and Benjamini and Schramm conjectured that  this property characterizes amenability. 

The whole setup and the conjecture can in fact be formulated in the more generall situation of locally finite infinite quasi-transitive graphs. In \cite{BS2}, the conjecture was proved for planar graphs. It was also shown that if $G$ is a locally finite quasi-transitive graph, then there exists a constant $k=k(G)$ such that the conjecture holds for the product of $G$ with a regular tree of degree higher than $k$. In this paper we will only work with Cayley graphs, and in this case the {\it Benjamini-Schramm non-unicity of percolation conjecture} says the following.

\begin{conjecture}[\cite{BS}]
If $\Gamma$ is a non-amenable group generated by a finite set $S$ then
$$p_c(\Gamma,S)<p_u(\Gamma,S).$$
\end{conjecture}

The only  non-amenable groups where the non-unicity of percolation conjecture is proved for all generating sets are groups with cost greater than $1$, as shown by R.Lyons (\cite{Ly00} see also \cite{Ly12}). For the approach to the Benjamini-Schramm conjecture through the theory of measurable equivalence relations see Gaboriau's paper \cite{Gb}.

Benjamini and Schramm showed that the critical value $p_c$ satisfies the following inequality (\cite{BS}):
$$p_c(\Gamma,S)\leq \frac{1}{\phi(\Gamma,S)+1}.$$
They also proved the following sufficient condition.

\begin{theorem}[\cite{BS}] \label{three} If 
\begin{equation}
\rho(\Gamma,S)p_c(\Gamma,S) |S|\ <\ 1,
\end{equation}
then  $p_c(\Gamma,S) < p_u(\Gamma,S)$. 
\end{theorem}

Theorem \ref{three} can be used to obtain further sufficient conditions  for $p_c(\Gamma,S) < p_u(\Gamma,S)$, as for example if combined with the estimate on $p_c(\Gamma,S)$ in terms of $\rho(\Gamma,S)$, $|S|$ and the girth of the graph, obtained in \cite{BNP}. Another corollary of (3) is

\begin{proposition}[\cite{PS}]\label{radius}
If $\rho(\Gamma, S)<\frac{1}{2}$, then $p_c(\Gamma,S)<p_u(\Gamma,S)$.
\end{proposition}

Pak and the second author derived from (3) the following weak version of Benjamini-Schramm conjecture for non-amenable groups.

\begin{theorem}[\cite{PS}]\label{multiset}
 For every non-amenable group $\Gamma$ and any symmetric finite generating set $S$ there exists a positive integer $k$ such that  $p_c(\Gamma,S^{(k)})<p_u(\Gamma,S^{(k)})$.
\end{theorem}

Note that $S^{(k)}$  stands here for the $k$-th power of the set $S\cup\{e_G\}$ understood as a {\it multiset}, so that the corresponding Cayley graph has lots of multiple edges. It is thus natural to look for a proof of the same result, but where only \lq\lq simple\rq\rq\ generating sets are allowed. Here {\it simple} means that the corresponding Cayley graph is a graph without loop or multiple edge.  In particular it would be desirable to have Theorem \ref{multiset} with the multiset $S^{(k)}$ replaced by the set $S^k = S\cdot S ... \cdot S \subset \Gamma$. From now on, by {\it generating set} we will only mean {\it simple generating set}.


It was observed already in \cite{PS} that  if $\Gamma$ contains a free group on two generators then there exist generating sets $\{A_l\}_{l\geq 1}$ such that $\rho(\Gamma,A_l)\rightarrow 0$ and thus for $l$ big enough we have $p_c(\Gamma,A_l)<p_u(\Gamma,A_l)$ by Proposition \ref{radius}. 

In view of its relation to the non-unicity of percolation problem, it is natural to inquire whether the spectral radius can always be made arbitrarily small in a finitely generated non-amenable group, in other words, whether non-amenability is equivalent to the following stronger property:

\begin{definition} Let $\Gamma$ be a non-amenable group. Suppose that there exists a sequence of finite generating sets $\{A_l\}_{l\geq 1}\subset \Gamma$ such that
$$\rho(\Gamma,A_l) \rightarrow 0, \ {\text as}\ l\rightarrow \infty .$$ 
We then say that $\Gamma$ has {\it infinitesimally small spectral radius} with respect to the family of generating sets $\{A_l\}_{l\geq 1}$.
\end{definition}

By Proposition \ref{radius}, having infinitesimally small spectral radius with respect to generators $\{A_l\}_{l\geq 1}$ implies that for big $l$, there is non-unicity of percolation on Cayley graphs of $(\Gamma,A_l)$.

\begin{question} \begin{enumerate}
\item Does every non-amenable group have infinitesimally small spectral radius with respect to some family of generating sets? 
\item With respect to some family of the form $\{S^k\}_{k\geq 1}$? 
\item With respect to $\{S^k\}_{k\geq 1}$ for any generating set $S$?
\end{enumerate}
\end{question} 



Note that if there is an element $g$ of infinite order in $\Gamma$, then the spectral radius can be made arbitrarily close to $1$ by adding to any given generating set bigger and bigger powers of $g$.
In Section \ref{balls} (Theorem \ref{balls_cond}), we show that the spectral radius is bounded away from $1$ uniformly on $S^k$, as $k\rightarrow \infty$, for arbitrary $S$ or, equivalently (by (2)), that that the conductance constant is bounded away from $0$ uniformly on $\{S^k\}$, $k\rightarrow \infty$.
In Section \ref{iso} we discuss a condition equivalent to non-amenability, which implies in particular that in a non-amenable group the isoperimetric constant can be made arbitrarily close to one. 

It is easy to check that if a subgroup of a group has infinitesimally small spectral radius then the same property holds for the group. As noted above, groups with free subgroups have infinitesimally small spectral radius. 
This result is extended to other classes of non-amenable groups in Section \ref{perc} below, where we show that the spectral radius goes to $0$ on certain sequences of generating sets in the group, and therefore the non-uniqueness of percolation conjecture holds at least on some generating sets, for groups with nontrivial non-amenable quotients (Theorem \ref{spec} and Corollary \ref{G/N} in Section \ref{perc}). This allows to make the same conclusion for infinite free Burnside groups (Corollary \ref{Burnside} in Section \ref{perc}) and for direct products $\Gamma\times \Bbb Z/d\Bbb Z$ with non-amenable $\Gamma$ and $d$ big enough (Corollary \ref{GtimesZ} in Section \ref{perc}). 

\smallskip





For a finitely generated non-amenable group, the question about the behaviour of the spectral radius on Cayley graphs of $\Gamma$ with respect to $S^k$ remains very much open. One sufficient condition for the spectral radius to be infinitesimally small on generators $S^k$ is Property (RD) of Jolissaint \cite{Jol}. 

\begin{proposition}[\cite{PS}] Suppose $\Gamma$ is a finitely generated non-amenable group with Property (RD). 
Then $\Gamma$ has infinitesimally small spectral radius with respect to the sequence $\{S^k\}_{k\geq 1}$ for any finite symmetric generating set $S$.
\end{proposition}

In Section \ref{perc}, Corollary \ref{<1}, we show that a group $\Gamma$ has infinitesimally small spectral radius with respect to the sequence $\{S^k\}_{k\geq 1}$ for a finite generating set $S$, if 
\begin{equation}
\frac{\rho(\Gamma,S) |S|}{gr(\Gamma,S)}<1 ,
\end{equation}
where $gr(\Gamma,S)$ denotes the rate of exponential growth of $\Gamma$ with respect to $S$.
Observe that the estimate $p_c(\Gamma,S)\leq gr(\Gamma,S)^{-1}$ combined with the condition (3) implies that (4) is enough to guarantee 
$p_c(\Gamma,S^k)<p_u(\Gamma,S^k)$ for big enough $k$. Our statement shows that $(4)$ implies in fact a stronger property, that $\rho(\Gamma,S^k) \rightarrow 0$ when $k\rightarrow\infty$. 

More generally, for any non-amenable quasi-transitive graph, it would be interesting to know how the spectral radius changes when edges are added to the graph so as to connect all vertices inside bigger and bigger balls. 

\begin{question} Let $G$ be a quasi-transitive locally finite non-amenable graph. For every $k\geq 1$, define $G_k$ by adding to $G$ edges connecting any two vertices at distance $\leq k$ in $G$. What is the asymptotics of $\rho(G_k)$ as $k\rightarrow\infty$?  Is it true that $\rho(G_k)\rightarrow 0$?
\end{question}

\noindent
{\bf Acknowledgment.}
The authors thank  Antoine Gournay, Nicolas Monod, Jesse Peterson and Alain Valette for helpful discussions and comments on the early draft of the paper. Antoine Gounay pouted out that Theorem \ref{crit} can also be deduced from the work of F\o lner, \cite{Fol}. 

One year after this work first appeared on Arxiv, Andreas Thom proved \cite{Thom} that the spectral radius of any non-amenable group can be arbitrarily small, thus answering our Question 6 (1).

\section{Spectral radius and non-unicity of percolation}\label{perc}

In this Section we investigate the property of infinitesimally small spectral radius and draw conclusions about non-unicity of percolation.

\begin{theorem}\label{spec}
If $\Gamma$ contains an infinite normal subgroup $N$ such  that $\Gamma/ N$ is non-amenable, then $\Gamma$ has infinitesimally small spectral radius.

\end{theorem}

\begin{proof}
Since $\Gamma/ N$ is non-amenable, there exists a generating set $S$ in $\Gamma/ N$ with $\rho(\Gamma/ N),S)\leq C<1$. 
Thus $\rho(\Gamma/ N, S^{(n)})=\frac{1}{|S|^n}||\Sigma_{g\in S^{(n)}}\lambda(g)||=(\rho(\Gamma/ N),S))^n\leq C^n$, where $S^{(n)}$ is the $n$-th power or $S$ understood as multiset. In particular, in $S^{(n)}$ each element $g\in S^n$ is included as many times as there exist freely reduced words of length at most $n$ in the alphabet $S$ representing $g$; denote this number $\alpha(g)$. We will now construct a lift of the generating multiset $S^{(n)}$ to a simple generating set 
in $\Gamma$, preserving the cardinality and controlling the spectral radius. Begin by choosing a lift  
$\hat{S}^n$ of the set $S^n$ to  $\Gamma$ such that $|S^n|=|\hat{S}^n|$. Now for each $g\in S^n$ choose $\alpha(g)$ different elements in $N$: $\{h_{g,1},\ldots, 
h_{g,\alpha(g)}\}$. Now consider
$$\bar{S_n}=\bigcup\limits_{\omega\in \hat{S}^n} \{ h_{g,1}g, h_{g,2}g, \ldots, h_{g,\alpha(g)}h \} .$$
The generating set $\bar{S_n}$ projects onto $S^{(n)}$ under the canonical projection, and  
thus we have $\rho(\hat{S}_n,\Gamma)\leq C^n$ which implies the statement.
\end{proof}

\begin{corollary}\label{G/N}
Let $\Gamma$ be a discrete group that contains an infinite normal subgroup $N$ such  that $\Gamma/ N$ is non-amenable. Then there exists a  finite set $S\subset \Gamma$ such that 
$$p_c(\Gamma,S)<p_u(\Gamma,S).$$
\end{corollary}

Let $B(n,p)$ be the free Burnside group on $n$ generators. By \cite{Ad} the group $B(n,p)$ is non-amenable for  $n \geq 2, p \geq 665$. We also have results of \cite{Sir} that show that for these groups  $B(2n,p)<B(n,p)$ and there exists a canonical quotient map from $B(2n,p)$ onto $B(n,p)\times B(n,p)$. 

\begin{corollary}\label{Burnside}
Let $B(n,p)$ be the free Burnside group with $n \geq 2, p \geq 665$. Then there exists a  finite set $S\subset B(n,p)$ such that 
$$p_c(B(n,p),S)<p_u(B(n,p),S).$$
\end{corollary}

As a direct modification of the proof of the Theorem \ref{spec} and the Proposition \ref{radius} we have the following.

\begin{corollary}\label{GtimesZ}
Let $\Gamma$ be a non-amenable finitely generated group. Then there exists $d=d(\Gamma,S)$ such that $p_c(\Gamma', S')<p_u(\Gamma', S')$ for $\Gamma'=\Gamma\times \mathbb{Z}/d\mathbb Z$ with a generating set  $S'$.
\end{corollary}

In a Cayley graph of a finitely generated group $\Gamma$ with respect to a generating set $S$, denote by $B_k(\Gamma,S)$ the ball of radius $k$, $k\in\mathbb{N}$.
Denote by $gr(\Gamma,S)=\lim\limits_{n\rightarrow \infty} |B_k(\Gamma,S)|^{\frac{1}{k}}$ the rate of exponential growth of $\Gamma$ with respect to the generating set $S$, strictly bigger than $1$ for any $S$ in a non-amenable $\Gamma$.

\begin{lemma}\label{power}
Let $\Gamma$ be a group generated by  a finite set $S$.  Then
$$\rho(\Gamma, S^k)\leq \frac{|S|^k}{|S^k|} (\rho(\Gamma, S))^k,$$
\end{lemma}
\begin{proof}
Let $X$ be a finite subset of $\Gamma$ and $A=\sum\limits_{g\in X}\beta_g \lambda(g)\in \mathbb{C}[\Gamma]$, then we have
\begin{align*}
\|A\| &=\lim_{p\rightarrow \infty} \tau((A^*A)^{p})^{\frac{1}{2p}},
\end{align*}
where $\tau$ is the standard trace on $\mathbb{C}[\Gamma]$, i.e. $\tau$ is a linear functional such that $\tau(g)=0$ if $g\neq 1$ and $\tau(1)=1$.
Therefore for every $\beta_g\in \mathbb{N}$, in particular if $\beta_g$ is the multiplicity $\alpha(g)$ of the element $g$ in the multiset $S^{(k)}$, we have
\begin{align*}||\sum\limits_{g\in S^k} \lambda(g)||\leq &||\sum\limits_{g\in S^k}\beta_g \lambda(g)||\\
=& ||\sum\limits_{g\in S^{(k)}}\lambda(g)||.
\end{align*}
Therefore we have
\begin{align*}
\rho(\Gamma, S^k)\leq& \frac{\rho(\Gamma,S^{(k)}) |S|^k}{|S^k|}\\
\leq& \frac{(\rho(\Gamma,S) |S|)^k}{|S^k|}.\\
\end{align*}
\end{proof}

\begin{corollary}\label{<1} Let $\Gamma$ be a non-ameanble group and $S$ a finite generating set in $\Gamma$, and let $gr(\Gamma,S)$ denote the rate of exponential growth of $\Gamma$ with respect to $S$. Assume that
$$\frac{\rho(\Gamma,S) |S|}{gr(\Gamma,S)} < 1 .$$
Then $\Gamma$ has  infinitesimally small spectral radius with respect to the family $\{S^k\}_k$ of generating sets: 
$$\rho(\Gamma,S^k)\rightarrow 0 \text{ when } k\rightarrow\infty .$$
\end{corollary}
\begin{proof}
Consider the generating sets $\{S^k\}_k$. The cardinality $|S^k|$ of $S^k$ is equal to the cardinality of the $k$-th ball in the Cayley graph of $\Gamma$ with respect to $S$, therefore the assumption implies that there exists  $k_0\in \mathbb{N}$ and $C<1$ such that for every $k\geq k_0$  we have $$\frac{\rho(\Gamma,S) |S|}{\sqrt[k]{|S^k|}}\leq C<1.$$
Therefore, by Lemma \ref{power}, $\rho(\Gamma, S^k)\rightarrow 0$, when $k\rightarrow\infty$.
\end{proof}

\section{Asymptotics of spectral radius and conductance constant along balls}\label{balls}

In this section we study the asymptotic behavior of the spectral radius and of the conductance constant on Cayley graphs with respect to the generating sets  
$S^k$, $k\geq 1$, for any $S$. Since both constants are related by inequalities (1) and (2), in our case it is sufficient to study only one of them.
\begin{theorem}\label{balls_cond}
Let $\Gamma$ be non-amenable group generated by a finite set $S$. Then there is a constant $0<C<1$ such that for every $k\in\mathbb{N}$ we have
$$\rho(\Gamma, S^k)\leq C.$$
In terms of conductance constant, there exists some $0<C'<1$ such that
$$h(\Gamma,S^k)\geq C'$$
for every $k\in\mathbb{N}$.
\end{theorem}
\begin{proof}
To reach a contradiction, assume that $\rho(\Gamma, S^k)\rightarrow 1$ on some subsequence, then  $h(\Gamma,S^k)\rightarrow 0$.  In terms of conductance constant this means that for every $\varepsilon>0$ there exists $k\in\mathbb{N}$ and a finite set $F\subseteq \Gamma$ such that
$$\sum\limits_{s\in B_{k}}|sF\backslash F|\leq\varepsilon |B_k||F|,$$ where, as above, $B_k$ denotes the ball of radius $k$ in the Cayley graph of $(\Gamma,S)$, i.e., $B_k=S^k$. 
Fix $\varepsilon$ and let $F$ and $k$ are given by inequality above. Consider the function $f=\sum\limits_{g\in B_k}\chi_{gF}$ and let $\Lambda_k$ be the sphere of radius $k$.
Then $\|f\|_{l_1(\Gamma)}=|F| |B_k|$. Note that for every $h\in S$ we have $|hB_k\backslash B_k|=|B_k\backslash hB_k|$ and $hB_k\backslash B_k\subseteq \Lambda_{k+1}$ and $B_k\backslash hB_k\subseteq \Lambda_k$. Thus for $h\in S$ we have 
\begin{align*}
\|\sum\limits_{g\in B_k}\chi_{hgF}-\sum\limits_{g\in B_k}\chi_{gF}\|_{1}&=\|\sum\limits_{g\in hB_k}\chi_{gF}-\sum\limits_{g\in B_k}\chi_{gF}\|_{1}\\
&= \|\sum\limits_{g\in hB_k\backslash B_k}\chi_{gF}-\sum\limits_{g\in B_k\backslash hB_k}\chi_{gF}\|_{1}\\
&=  \|\sum\limits_{g\in hB_k\backslash B_k}(\chi_{gF}-\chi_F)-\sum\limits_{g\in B_k\backslash hB_k}(\chi_{gF}-\chi_F)\|_{1}\\
&\leq \sum\limits_{g\in \Lambda_{k+1}}\|\chi_{gF}-\chi_F\|_1+\sum\limits_{g\in\Lambda_{k}}\|\chi_{gF}-\chi_F\|_1\\
&=\sum\limits_{g\in \Lambda_{k+1}}|gF\Delta F|+\sum\limits_{g\in\Lambda_{k}}|gF\Delta F|\\
& \leq \sum_{g\in B_{k+1}} |gF\Delta F|\\
& \leq 2\varepsilon |B_{k+1}| |F|\\
&\leq 2\varepsilon \frac{|B_{k+1}|}{|B_k|} \|f\|_1\\
&\leq 2\varepsilon |S|  \|f\|_1.
\end{align*}
Normalizing $f$ we obtain a positive function $f\in l_1(\Gamma)$ with $\|f\|_1=1$ and $\|hf-f\|\leq \varepsilon'$ for every $h\in S$. Thus $\Gamma$ is amenable.
\end{proof}


\section{Isoperimetric constants of non-amenable groups}\label{iso}

In this section we discuss a characterization of amenability related to the property of having infinitesimally small spectral radius. As an application we provide a new estimate on the isoperimetric constant of the group.\\

We will need the following lemma which is well known (see e.g. Proposition 11.5 in \cite{BHV}), but we include it for completeness.
\begin{lemma}\label{inv}
Let $\pi:\Gamma\rightarrow B(H)$ be a unitary representation of a discrete group $\Gamma$. Suppose that there exists a unit vector $\xi\in H$ such that $\|\pi(g)\xi-\xi\|\leq C<\sqrt{2}$ for every $g\in \Gamma$. Then $\pi$ has an invariant vector.
\end{lemma}
\begin{proof}
Note that 
\begin{align*}
Re(<\pi(g)\xi,\xi>)&= 1- \frac{1}{2} ||\pi(g)\xi-\xi||^2\\
& \geq 1-\frac{C^2}{2}=C'>0.
\end{align*}
Let $V=\overline{conv\{\pi(g)\xi: g\in \Gamma\}}$ then $V$ is $\pi(\Gamma)$-invariant and $$Re(<\theta,\xi>)\geq C' \text{ for every } \theta\in V.$$
Let $\nu\in V$ be the unique element of $V$ that has minimal norm, then $Re(<\nu,\xi>)\geq C'$ and $\nu\neq 0$. Since $\pi$ is a unitary representation, $\nu$ is invariant under the action of $\pi(\Gamma)$.
\end{proof}

\begin{theorem}\label{crit}
A  finitely generated group $\Gamma$ is amenable if and only if there exists a constant $C<2$ such that for every finite set $S\subset \Gamma$ there exists a finite set $F\subset \Gamma$ such that 
$$|sF\Delta F|\leq C |F|,\text{ for every } s\in S.$$ 
\end{theorem}
\begin{proof}
The existence of $C\leq 2$ that satisfy the condition of the theorem for amenable group $\Gamma$ follows from F\o lner's criteria. 

To prove the converse fix a finite set $S$ and let $F$ be a finite set of $\Gamma$ such that 
$$|sF\Delta F|\leq C |F|,\text{ for every } s\in S.$$ 

Consider $\xi_F=\frac{1}{\sqrt{|F|}}\chi_{F}$, we have $||\lambda(s)\xi_F-\xi_{F}||\leq \sqrt{C}$ for every $s\in S$. Let  $S_i$ be  an increasing sequence of sets in $\Gamma$ with $\Gamma=\cup S_i$ and let  $\lambda_{\omega}:G\rightarrow B(l_2(\Gamma)^{\omega})$ be an ultra-limit of the left-regular representation acting on an ultra power of the Hilbert space $l_2(\Gamma)$. Then for the vector $\xi=(\xi_{F_i})_{i\in\mathbb{N}}$ we have that $||\lambda_{\omega}(g)\xi-\xi||\leq \sqrt{C}$ for every $g\in G$. By Lemma \ref{inv} we have that $\lambda_{\omega}$ has an invariant vector. Thus $\lambda$ has a sequence of almost invariant vectors, therefore $\Gamma$ is amenable. 
\end{proof}

As a direct application of the Theorem \ref{crit} we have the following corollary.

\begin{corollary}
Let $\Gamma$ be a non-amenable group then for every $\varepsilon>0$ there exists a finite set $S\subset \Gamma$ such that 
$$\phi(\Gamma,S)\geq 1-\varepsilon.$$
\end{corollary}

\begin{remark} If the condition equivalent to amenability in Theorem \ref{crit} could be strengthen to 
say that there exists a constant $C<2$ such that for every finite set $S\subset \Gamma$ there exists a finite set $F\subset \Gamma$ such that 
$$\Sigma_{s\in S} |sF\Delta F|\leq C |F| |S|,$$
then it would imply that for every non-amenable group the conductance constant is arbitrary close to $1$ on some finite sets and thus the group has infinitesimally small spectral radius, by Mohar's inequalities.
 \end{remark}


\begin{thebibliography}{99}

\bibitem{Ad}\textsc{S. I. Adyan,} {\it Random walks on free periodic groups.} \textrm{Izv. Akad. Nauk SSSR Ser. Mat. 46 (1982), 1139Ð1149, 1343.}

\bibitem{BHV}\textsc{B. Bekka, P. de la Harpe, A. Valette,}  Kazhdan Property (T). Cambridge University Press, 2008.

\bibitem{BNP}\textsc{I. Benjamini, A. Nachmias, Y. Peres,} {\it Is the critical percolation probability local?}  \textrm{Probab. Theory Related Fields 149 (2011), no. 1-2, 261–-269.}

\bibitem{BS}\textsc{I. Benjamini, O. Schramm,} {\it Percolation beyond $\mathbb{Z}^d$, many questions and a few answers. } \textrm{Electron. Comm. Probab. 1 (1996), no. 8, 71--82} 

\bibitem{BS2}\textsc{I. Benjamini, O. Schramm,} {\it Percolation in the hyperbolic plane.} \textrm{ J. Amer. Math. Soc. 14 (2001), no. 2, 487--507}

\bibitem{BK} \textsc{R. Burton, M. Keane}, {\it Density and uniqueness in percolation.} \textrm{Comm. Math. Phys. 121 (1989), 501--505}

\bibitem{Dodziuk}\textsc{J. Dodziuk,} {\it Difference equations, isoperimetric inequality and transience of certain random walks.}  \textrm{Trans. Amer. Math. Soc. 284 (1984), no. 2, 787--794.}

\bibitem{Fol}\textsc{E. F\o lner,} {\it On groups with full Banach mean value.} \textrm{Math. Scand. 3 (1955),  243--254.}

\bibitem{Gb}\textsc{D. Gaboriau,} {\it Invariant percolation and harmonic Dirichlet functions.} \textrm{Geom. Funct. Anal. 15 (2005), no. 5, 1004–-1051.}


\bibitem{HP} \textsc{O. H\"aggstr\"om,  Y. Peres,} {\it Monotonicity of uniqueness for percolation on Cayley graphs: All
    infinite clusters are born simultaneously.} \textrm{Probab. Th. Rel. Fields, 113 (1999), 273--285.}

\bibitem{Jol} \textsc{P. Jolissaint,} {\it Rapidly decreasing functions in reduced $C^*$-algebras of groups.} \textrm{Trans. Amer. Math. Soc. 317 (1990), 167--196.}

\bibitem{Kesten}\textsc{H. Kesten,} {\it Symmetric random walks on groups.} \textrm{Trans. Amer. Math. Soc. 92 (1959), 336--354.}

\bibitem{Ly95} \textsc{R. Lyons,} {\it Random walks and the growth of groups.} \textrm{C. R. Acad. Sci. Paris S´er. I Math. 320,
1361–1366.}

\bibitem{Ly00} \textsc{R. Lyons,} {\it Phase transitions on nonamenable graphs.} \textrm{J. Math. Phys. 41 (2000), 1099-1126.}

\bibitem{Ly12}\textsc{R. Lyons,} {\it Fixed Price of Groups and Percolation.} \textrm{arXiv:1109.5418, to appear in Erg. Th. Dyn. Syst.}

\bibitem{Mohar}\textsc{B. Mohar,} {\it Isoperimetric inequalities, growth, and the spectrum of graphs.} \textrm{Linear Algebra Appl. 103 (1988), 119--131.}

\bibitem{PS}\textsc{I. Pak, T. Smirnova-Nagnibeda,} {\it On non-uniqueness of percolation on non-amenable Cayley graphs.} \textrm{C. R. Acad. Sci. Paris S\'er. I Math. 330 (2000), no. 6, 495--500.}

\bibitem{Sir}\textsc{V. L. Sirvanjan,} {\it Imbedding of the group $B(\infty,n)$ in the group $B(2,n)$. Izv. Akad. Nauk SSSR Ser. Mat. 40 (1976), 190--208, 223.}

\bibitem{Thom}\textsc{A. Thom,} {\it A remark about the spectral radius.} ArXiv:1306.1767

\end{thebibliography}
\end{document}